\documentclass{amsart}
\usepackage{amssymb}
\usepackage{amsmath,amscd}
\usepackage{combelow}
\usepackage{geometry}
\newcommand{\arr}{\rightarrow}

\newcommand{\R}{\mathbb{R}}

\newcommand{\vol}{\omega^{\wedge n}}
\newcommand{\A}{\mathcal{A}}

\newcommand{\Symp}{\operatorname{Symp}(M,\omega)}
\newcommand{\Diff}{\operatorname{Diff}}
\newcommand{\Homeo}{\operatorname{Homeo}}

\newtheorem{theorem}{Theorem}
\newtheorem{corollary}[theorem]{Corollary}

\newtheorem*{question*}{Question}
\newtheorem{definition}[theorem]{Definition}
\newtheorem{lemma}[theorem]{Lemma}
\newtheorem{proposition}[theorem]{Proposition}

\newtheorem*{lemma*}{Lemma}
\newtheorem*{theorem*}{Theorem}
\newtheorem*{remark*}{Remark}
\newtheorem*{definition*}{Definition}
\newtheorem{remark}[theorem]{Remark}
\theoremstyle{remark}

\theoremstyle{definition}

\newtheorem*{claim*}{Claim}

\begin{document}

\title[No symplectic-Lipschitz structures on $S^{2n \geq 4}$]{No symplectic-Lipschitz structures on $S^{2n \geq 4}$}
\author{ Du\v{s}an Joksimovi\'c}
\address{Sorbonne Universit\'e, Universit\'e de Paris, CNRS, Institut de Math\'ematiques de Jussieu-Paris Rive Gauche, F-75005 Paris, France}
\email{joksimovic@imj-prg.fr}

\begin{abstract}
We prove that a closed manifold which admits a symplectic-Lipschitz structure has non-vanishing even-degree cohomology groups with real coefficients.  In particular, spheres $S^{2n}, n \geq 2$ do not admit symplectic-Lipschitz structures.
\end{abstract}

\maketitle

\section{Introduction and main results}

A smooth symplectic manifold is a smooth manifold $M$ equipped with a closed and non-degenerate 2-form $\omega.$
A \emph{symplectic diffeomorphism} (i.e. \emph{symplectomorphims}) is a diffeomorphism $\varphi: (M,\omega) \arr (M,\omega)$ which satisfies that $\varphi^*\omega = \omega.$ We denote by $\Symp$ the set of all symplectomorphisms of $(M,\omega).$

By a theorem due to Y. Eliashberg and M. Gromov it is known that $\Symp$ is closed w.r.t. $C^0$-topology (i.e. compact-open topology) inside the group of all diffeomorphisms $\Diff(M)$ (for a proof see e.g. \cite[Theorem 12.2.1]{MS}). This result is considered as the beginning of a subfield of symplectic geometry called \emph{$C^0$-symplecic geometry}. Roughly, it investigates which symplectic phenomena persist under $C^0$-limits. The central objects are 
\emph{symplectic homeomorphisms} which are elements of 
the closure   of $\Symp$ in the compact-open topology inside the group of all homeomorphisms $\Homeo(M).$ The notion of symplectic homeomorhpsims leads to the notion of symplectic topological manifolds.

\begin{definition*}[Symplectic topological manifolds]
\emph{A symplectic topological atlas} on a topological $2n$-dimensional manifold $M$ is a topological atlas $\{(U_i,\phi_i)\}_{i \in I}$ whose transition functions are symplectic homeomorphisms. A \emph{topological symplectic manifold} (or \emph{$C^0$-symplectic manifold}) is a topological manifold equipped with a maximal symplectic topological atlas.\footnote{Authors usually use the terms $C^0$-symplectic atlas and $C^0$-symplectic manifold. We decided to call them differently since we want to emphasize that we study topological manifolds with additional structure, and not smooth manifolds equipped with a certain continuous differential form.}  
\end{definition*}

Every smooth symplectic manifold trivially admits a symplectic topological structure. This naturally raises a question whether there are manifolds which admit a symplectic topological structure, but not a smooth symplectic structure.
The following question due to H. Hofer is one of the central open questions in $C^0$-symplecic geometry.

\begin{question*}[Hofer '90s]
Does $S^{2n \geq 4}$ admit a symplectic topological structure?
\end{question*}


The main result of this article answers the question in the negative if we replace the word ``topological'' by ``Lipschitz''.
More precisely, we prove that manifolds which admit weak symplectic-Lipschitz structures have non-vanishing even cohomology groups with real coefficients. 

To state the main result we first introduce the notions of Lipschitz manifolds and (weak) symplectic-Lipschitz manifolds.  

\begin{definition}[Lipschitz maps] \label{lip-map-def}
A map $f:(X,d_X) \arr (Y,d_Y)$ between metric spaces $(X,d_X)$ and $(Y,d_Y)$ is \emph{Lipschitz} if there exists $L \geq 0$ such that $$d_Y(f(x),f(x')) \leq Ld_X(x,x'),$$ for all $x,x' \in X.$
It is called \emph{bi-Lipschitz} if it is a homeomorphism and both $f$ and $f^{-1}$ are Lipschitz. 
\end{definition}

\begin{definition}[Lipschitz manifolds] \label{lip-manifold-def}
\emph{A Lipschitz manifold} $M$ of dimension $n$ is a second-countable locally
compact Hausdorff space $M$ equipped with a family of (Lipschitz) coordinate
charts $\{(U_a, \phi_a)\}_{a \in \A}$, satisfying the following conditions:
\begin{itemize}
    \item[(i)] sets $U_a, a \in \A$ form an open cover of $M,$
    \item[(ii)] each $\phi_a: U_a \arr \R^n$ is a homeomorphism onto its image, and
    \item[(iii)] the transition functions
    \begin{equation} \label{transition-maps}
        \Phi_{ab} := (\phi_b \circ \phi_a^{-1}) \vert_{\phi_a(U_a \cap U_b)}: \phi_a(U_a \cap U_b )  \arr \phi_b(U_a \cap U_b)
    \end{equation}
are locally bi-Lipschitz (with respect to the standard metric on $\R^n$).    
\end{itemize}

\end{definition}

The main objects of this article are the following.

\begin{definition}[Symplectic-Lipschitz manifolds] \label{lip-symp-def}
A manifold $M$ of dimension $2n$ is \emph{symplectic-Lipschitz} if it admits a Lipschitz atlas $\{(U_a, \phi_a)\}_{a \in \A}$ such that 
the transition maps \eqref{transition-maps} preserve the standard symplectic structure on $\R^{2n}$, i.e. 
${\Phi_{ab}}^*\omega_0 = \omega_0$ (Lebesgue) a.e.
Then the collection $\omega := \{\omega_0 \vert_{\phi_a(U_a)}\}_{a \in \A}$ is called a \emph{symplectic-Lipschitz form} (or \emph{symplectic-Lipschitz structure}) on $M.$
\end{definition}

The reason why Definition \ref{lip-symp-def} makes sense is mainly due to Rademacher's theorem (see Theorem \ref{rademacher} below) which allows us to define the pullback of a differntial form on $\R^{2n}$ by a bi-Lipschitz map (Lebesgue) a.e.. 
%
Moreover, the nice behaviour of bi-Lipschitz functions allows us to define an analog of differential forms (and standard operations on them) on Lipschitz manifolds, so-called \emph{flat forms}. In short, those are forms which in charts are represented by differential forms with coefficients in $L^{\infty}$ and which have the exterior derivative (in the sense of distributions) which is also an $L^{\infty}$ form. For more details we refer the reader to Section \ref{section-lip}. 

We can also extend the notion of non-degeneracy to flat 2-forms. Namely, a flat 2-form on an $2n$-dimensional Lipschitz manifold is called \emph{non-degenerate} if its $n$-th wedge power is represented in charts by an expression of the form $f dx_1 \wedge ... \wedge dx_{2n},$ where $f$ is an $L^{\infty}$ function s.t. $f>0$ a.e. (see Definition \ref{non-deg-form} in Section \ref{section-lip}).
These observations lead us to the following generalization of smooth symplectic manifolds.

\begin{definition}[Weak symplectic-Lipschitz manifolds] \label{weak-lip-mfld}
\emph{A weak symplectic-Lipschitz manifold} is a
Lipschitz manifold $M$ equipped with a closed (in the sense of distributions) and non-degenerate flat 2-form.
\end{definition}

Let us note that every (weak) symplectic-Lipschitz manifold is orientable, see Remark \ref{orientability} on p. \pageref{orientability}.

\begin{remark*}[Symplectic-Lipschitz manifolds] \label{rmk-symp-lip-mfld} \normalfont
We proposed two different definitions of symplectic-Lipschitz manifolds.
Examples of weak symplectic-Lipschitz manifolds (see Definition \ref{weak-lip-mfld}) include all symplectic-Lipschitz manifolds (see Definition \ref{lip-symp-def}). For the proof see Proposition \ref{lip-symp-equiv} in Section \ref{section-proof}.
Up to our knowledge it is not known whether they are equivalent, i.e. whether Darboux theorem holds in the setting of Lipschitz manifolds and flat forms. According to the literature, this question was raised first by D. Sullivan, see p. 65 in \cite{Hainonen-lecture-notes}.
\end{remark*}

We are now ready to state the main result of this article.

\begin{theorem} \label{main-thm}
Let $(M,\omega)$ be a closed (compact and without boundary) weak symplectic-Lipschitz manifold of dimension $2n.$ Then the singular cohomology group $H^{2k}(M, \R) \neq 0,$ for every $k \in \{1,2,...,n\}.$
\end{theorem}

An immediate corollary of Theorem \ref{main-thm} is the following.

\begin{corollary}
For $n \geq 2,$ the sphere $S^{2n}$ does not admit a (weak) symplectic-Lipschitz structure. 
\end{corollary}

The proof exploits the fact that Lipschitz manifolds have well-defined de Rham complex (the complex of \emph{flat forms}) which shares similar properties as de Rham complex in the smooth case (see Section \ref{section-lip}). 
After noticing this fact the proof goes analogously as in the smooth case by proving that classes of the wedge powers of the symplectic-Lipschitz form are not trivial.

\begin{remark*}[Topological vs Lipschitz manifolds] \label{c0-vs-lip} \normalfont
Categories of topological and Lipschitz manifolds are almost equivalent in dimensions other than 4.
More precisely, in \cite{Sullivan} D. Sullivan proved that every topological manifold of dimension $n \neq 4$ admits a unique Lipschitz structure. In this sense the Lipschitz condition from Theorem \ref{main-thm} a priori does not impose strong restrictions.

The analogous result between the categories of Lipschitz and smooth manifolds does not hold since not every Lipschitz manifold is smoothable (see e.g. \cite[Theorem 1.1]{Heinonen-Keith}). 
Notice that such a result, by Sullivan's theorem, would imply that every topological manifold is smoothable.   
\end{remark*}

The remark above motivates the question whether Sullivan's theorem holds in the symplectic category (maybe even in dimension 4), i.e. whether every symplectic topological manifold admits a (unique) symplectic-Lipschitz structure.  
In the case that it holds (even without the uniqueness part), together with Theorem \ref{main-thm}, it would imply a negative answer to the question by H. Hofer.


\subsection*{Organization} In Section 2 we give an overview on results from analyis on Lipschitz manifolds that we are going to use in the proof of Theorem \ref{main-thm}. Section \ref{section-proof} is devoted to the proof of Theorem \ref{main-thm}.

\subsection*{Acknowledgement} I would like to thank Fabian Ziltener for his advise to look at Lipschitz category as an intermediate step in understanding $C^0$-symplectic structures as well as for many inspiring discussions. 
I would also like to thank Sobhan Seyfaddini for an interesting discussion, and Dennis Sullivan for sharing his expertise on Lipschitz manifolds.
The work on this project was funded by ERC Starting Grant No. 851701: homeomorphisms in symplectic topology and dynamics (HSD).

\section{Analysis on Lipschitz manifolds} \label{section-lip}

In this section we recall some of the standard results from the analysis of Lipschitz functions and its applications to Lipschitz manifolds. 
For more details on the subject we refer the reader to \cite{Sullivan-Donaldson,Goldstein,Hainonen-lecture-notes,Rosenberg,Sullivan,Sullivan-survey,Teleman-book,Teleman, Whitney} and references therein.

One of the main reasons which allows us to apply many analytical tools to Lipschitz functions is that they behave similar to $C^1$-functions. This is the statement of the following classical result (see also \cite[Theorem 5.20]{Teleman-book}).

\begin{theorem}[Rademacher's theorem] \label{rademacher}
Let $U,V \subseteq \R^n$ be open sets and $f: U \arr V$ a Lipschitz map. Then for all $1 \leq i \leq n$ the following hold:
\begin{itemize}
    \item[(i)] partial derivatives $\frac{\partial f}{\partial x_i}$ exist (Lebesgue) almost everywhere on $U$ in the strong sense, and they are also weak derivatives of $f,$
    \item[(ii)] $\frac{\partial f}{\partial x_i} \in L^{\infty}(U)$ w.r.t. Lebesgue measure.
\end{itemize}
\end{theorem}

As a consequence of Theorem \ref{rademacher} the Jacobi matrix $J_f$ of a given Lipschitz function is well-definied a.e. and its elements are of class $L^{\infty}.$

We first define $L^{\infty}$ forms on $\R^n$ which serve as a local model (i.e. form in charts) for $L^{\infty}$ forms on Lipschitz manifolds. 

\begin{definition}[$L^{\infty}$ forms on $\R^n$] \label{Rn-forms}
An expression of the form
$$\alpha = \sum_I a_I dx^I := \sum_{1 \leq i_1 < ... < i_k \leq n} a_{i_1,...,i_k} dx_{i_1} \wedge ... \wedge dx_{i_k},$$
is called an \emph{$L^{\infty}$ $k$-form} if $a_I := a_{i_1,...,i_k} \in L^{\infty}(\R^n),$ for all multi-indices $I=(i_1,...,i_k).$
\end{definition}

For example, by Theorem \ref{rademacher} we have that a Lipschitz map $f: \R^n \arr \R$ induces an $L^{\infty}$ 1-form
$$df := \sum_{i=1}^n \frac{\partial f_i}{\partial x_i} dx_i.$$

Next we explain how one can extend the standard operations on differential forms (pullback, exterior differential, and wedge product) to $L^{\infty}$ forms. After that we will use those definitions to define $L^{\infty}$ de Rham complex on Lipschitz manifolds.

The wedge product of two $L^{\infty}$ forms is defined analogously to the smooth case.

\begin{definition}[Wedge product of $L^{\infty}$ forms]
Let $\alpha = \sum_I a_I dx^I$ be an $L^{\infty}$ $k$-form and $\beta = \sum_J b_J dx^J$ be an $L^{\infty}$ $l$-form on $U \subseteq \R^n.$ \emph{The wedge product} of $\alpha$ and $\beta$ denoted by $\alpha \wedge \beta$ is the $L^{\infty}$ $(k+l)$-form defined as follows
$$\alpha \wedge \beta := \sum_{I,J} (a_I b_J) dx^I \wedge dx^J.$$
\end{definition}

The following notions are used to define weak symplectic-Lipschitz manifolds.

\begin{definition}[$L^{\infty}$ volume forms] \label{volume-form}
A top degree $L^{\infty}$ form on $U \subseteq \R^n$ is called an \emph{$L^{\infty}$ volume form} on $U$ if it is of the form $f dx_1 \wedge ... dx_n,$ for some $f \in L^{\infty}(U)$ which satisfies $f >0$ a.e. on $U.$
\end{definition}

\begin{definition}[Non-degenerate $L^{\infty}$ 2-forms] \label{non-deg-form}
An $L^{\infty}$ 2-form $\alpha$ on $U \subseteq \R^{2n}$ is called \emph{non-degenerate} if $\alpha^{\wedge n}$ is an $L^{\infty}$ volume form.
\end{definition}

Next we define the pullback by a bi-Lipschitz map. 
For that we will need the following lemma. The proof is standard and is left to the reader.

\begin{lemma} \label{zero-measure-lemma}
Let $U$ and $V$ be open sets in $\R^n$ and $f: U \arr V$ a Lipschitz map. Then $f$ maps a set of Lebesgue measure zero to a set of Lebesgue measure zero.
\end{lemma}

\begin{definition}[Pullback by a bi-Lipschitz function]
Let $U \subseteq \R^n$ and $V \subseteq \R^m$ be open sets, 
\begin{equation*}
f: U \arr V, \quad x:=(x_1,...,x_n) \mapsto f(x) := (f_1(x), ..., f_m(x))  
\end{equation*}
a bi-Lipschitz map, and $\alpha = \sum_{I} a_{I} dx^I$ an $L^{\infty}$ $k$-form on $V.$ 
We define $f^*\alpha$ to be the $L^{\infty}$ $k$-form on $U$ given by
\begin{equation}\label{pullback}
f^*\alpha := \sum_{I} (a_{I} \circ f) \hspace{1mm} df^I,    
\end{equation}
where $df^I := df_{i_1} \wedge ... \wedge df_{i_k}$ for a multi-index $I=(i_1,...,i_k).$
\end{definition}

Let us note that the pullback is well-defined for the following two reasons:
\begin{itemize}
    \item $a_I \circ f$ are defined a.e. on $U.$ This follows from Lemma \ref{zero-measure-lemma} and the fact that $f$ is bi-Lipschitz,  
    
    \item $a_I \circ f \in L^{\infty}(U).$ 
\end{itemize}

\begin{remark*} \normalfont
It is possible to define the pullback by a function which is just Lipschitz, but in that case one should take care of the fact that $f$ can map an open set into a set where the $a_I$'s are not defined. For more details we refer the reader to \cite[Theorem 5.24]{Hainonen-lecture-notes}. For our purposes it is enough to define the pullback by a bi-Lipschitz map since we are going to work with Lipschitz manifolds.
\end{remark*}

Next we define the exterior derivative of $L^{\infty}$ forms. 

\begin{definition}[Exterior derivative of $L^{\infty}$ forms] \label{ext-dif}
Let $\alpha$ be an $L^{\infty}$ $k$-form on an open set $U \subseteq \R^n.$ Then $\alpha$ is said to have \emph{distributional exterior derivative} if there exists a $(k+1)$-form $d \alpha$ with measurable coefficients on $U$ such that for every smooth $(n-k-1)$-form $\sigma$ with compact support in $U$ the following holds:
$$\int_U \alpha \wedge d\sigma = (-1)^{k+1} \int_U d\alpha \wedge \sigma.$$
\end{definition}

The following proposition is crucial for defining a chain complex of flat forms, see \cite{Teleman} for the proof. 

\begin{proposition} \label{d^2=0} \normalfont
The exterior derivative satisfies $d \circ d = 0.$   
\end{proposition}

The following notion of \emph{flat forms}, introduced by H.~Whitney \cite{Whitney}, is central in defining de Rham cohomology on Lipschitz manifolds. 

\begin{definition}[Flat forms]
An $L^{\infty}$ form $\alpha$ on $U \subseteq \R^n$ is called \emph{flat} if it has distributional exterior derivative $d\alpha$ which is an $L^{\infty}$ form on $U$. 
We denote by $\Omega_{flat}^k (U)$ the set of all flat $k$-forms on $U,$ and by $\Omega_{flat}^* (U)$ the set of all flat forms on $U.$ 
\end{definition}

\begin{remark*} \normalfont
Notice that by Proposition \ref{d^2=0} all \textit{exact} $L^{\infty}$ forms are flat. 
\end{remark*}

\begin{proposition} \label{comutativity}
Let $U, V \subseteq \R^n$ be open sets with compact closure, $f: U \arr V$ a bi-Lipschitz map and $\alpha, \beta$ be $L^{\infty}$ forms. Then the following hold:
\begin{align*}
    f^* d\alpha &= d (f^*\alpha), \\
    f^* (\alpha \wedge \beta) &= f^* \alpha \wedge f^* \beta.
\end{align*}

\end{proposition}

\begin{proof}
See \cite[p. 155]{Goldstein}. 
\end{proof}

We are now able to define $L^{\infty}$ and flat forms on Lipschitz manifolds. 
For that let $M$ be an $n$ dimensional oriented and closed Lipschitz manifold, and  $\{(U_a, \phi_a)\}_{a \in \A}$ its (maximal, oriented) Lipschitz atlas.

\begin{definition}[$L^{\infty}$ forms on Lipschitz manifolds]
\emph{An $L^{\infty}$ $k$-form $\alpha$ on $M$} is a collection
$$\alpha := \{\alpha_{a} \}_{a \in \A},$$ such that: 
\begin{itemize}
    \item[(i)] $\alpha_a$ is an $L^{\infty}$ $k$-form on $\phi_a(U_a) \subseteq \R^n,$
    \item[(ii)] for every $a,b \in \A$ it holds that ${\Phi_{a,b}}^*\alpha_b = (\phi_b \circ \phi_a^{-1})^*\alpha_b = \alpha_a.$
\end{itemize}
\end{definition}

Definition \ref{ext-dif} naturally extends to $L^{\infty}$ forms on Lipschitz manifolds.
Namely, we say that $\alpha = \{\alpha_a\}_{a \in \A}$ has the \textit{exterior derivative} $d\alpha$ if for every $a\in \A$ the form $\alpha_a$ has the exeterior derivative $d\alpha_a.$
By Proposition \ref{comutativity} it follows that the exterior derivative $d$ is well-defined, i.e.
$d\alpha = \{ d \alpha_a\}_{a \in \A}.$
Moreover, Proposition \ref{d^2=0} and \ref{comutativity} imply that  $d$ satisfies $d \circ d =0.$ 
In the same manner we can extend Definitions \ref{volume-form} and \ref{non-deg-form} to Lischitz manifolds.

\begin{definition}[Flat forms on Lipschitz manifolds]
An $L^{\infty}$ $k$-form $\alpha := \{\alpha_{a} \}_{a \in \A}$ on $M$ is called \emph{flat} if $\alpha_a$ is flat on $\phi_a(U_a)$ for all $a \in \A.$ We denote by $\Omega_{flat}^k (M)$ the set of all flat $k$-forms on $U,$ and by $\Omega_{flat}^* (M)$ the set of all flat forms on $U.$ 
\end{definition}

\begin{remark}[Orientability of Lipschitz manifolds equipped with an $L^{\infty}$ volume form] \label{orientability} \normalfont
Let $M$ be a Lipschitz manifold equipped with an $L^{\infty}$ volume form $\omega.$
Let $\{(U_a, \phi_a)\}_{a \in \mathcal{A}}$ be an atlas s.t. in every chart $\omega$ is represented by $\omega_a := f_a dx_1 \wedge ... \wedge d x_{2n}$ for some $ f_a \in L^{\infty}(\phi_a(U_a))$ and $f_a>0$ a.e. on $\phi_a(U_a).$
Consider two charts $(U_i, \phi_i), i =a,b$ which have a non-empty intersection. 
Then we have 
$\Phi_{ab}^* \omega_b = \omega_a$ a.e. on $\phi_a(U_a \cap U_b)$ or equivalently $f_a = (f_b \circ \Phi_{ab}) \operatorname{Jac} (\Phi_{ab})$ a.e. on $\phi_a(U_a \cap U_b).$ Since $f_a, f_b>0$ a.e. it follows that $\operatorname{Jac} (\Phi_{ab})>0$ a.e. on $\phi_a(U_a \cap U_b).$
Hence $\Phi_{ab}$ is orientation preserving and therefore $\{(U_a, \phi_a)\}_{a \in \mathcal{A}}$ is an oriented atlas.
As a consequence we conclude that weak symplectic-Lipschitz manifolds are orientable. 
\end{remark}

\begin{definition}
We define \emph{$L^{\infty}$ de Rham cohomology of a Lipschitz manifold $M$} to be the
cohomology of the complex $(\Omega_{flat}^*(M), d)$ and we denote it by $H_{L^{\infty},dR}^*(M).$ 
\end{definition}

The following theorem connects de Rham cohomology with the singular cohomology of $M.$

\begin{theorem} \label{isomorphic}
$L^{\infty}$ de Rham cohomology is isomorphic to the singular cohomology with the real coefficients, i.e.
$H_{L^{\infty},dR}^* (M) \cong H^*(M, \R).$
\end{theorem}

\begin{proof}[Proof of Theorem \ref{isomorphic}]
Follows from  Theorem 12A (p. 230), Theorem 7C (p. 265), IX.12 (p. 272-274), and IX.14 (p. 276-280) from the Whitney's book \cite{Whitney}.

For a proof in the concrete situation of Lipschitz manifolds we refer the reader to \cite[de Rham's Theorem, p. 158]{Goldstein}, \cite[Theorem 2.1]{Teleman}, and \cite[Theorem 5.21]{Teleman-book}.
\end{proof}

Moreover, we have a form of Poincar\'e duality for $L^{\infty}$ de Rham cohomology. 
To make it precise we need to define the integration of $L^{\infty}$ $n$-forms over $M.$
For that we need the following.

\begin{definition}[Lipschitz partition of unity] \label{def-part-unity}
A \emph{Lipschitz partition of unity subordinated to the cover} $(U_a)_{a \in \A}$ is a family $(\rho_a)_{a \in \A}$ of Lipschitz maps $\rho_a: M \arr [0,1]$ such that the following hold:
\begin{itemize}
    \item[(i)] supports $\operatorname{supp} \rho_a$ of $\rho_a$ form a locally finite family,\footnote{A collection of subsets of a topological space $X$ is said to be locally finite, if each point in the space has a neighbourhood that intersects only finitely many of the sets in the collection.}
    \item[(ii)] $\operatorname{supp} \rho_a \subseteq U_a$, for all $a \in \A$,
    \item[(iii)] $\sum_{a \in \A} \rho_a(x) = 1,$ for all $x \in M.$
\end{itemize}
\end{definition}

\begin{lemma} \label{partition-unity}
Let $M$ be a Lipschitz manifold and $\mathcal{U} = (U_a)_{a \in \A}$ an open cover of $M.$ Then there exists a Lipschitz partition of unity subordinated to $\mathcal{U}.$  
\end{lemma}

\begin{proof}[Proof of Lemma \ref{partition-unity}]
Follows from \cite[Theorem 3.5 and Theorem 5.3]{Vaisala}.
\end{proof}

Choose a finite cover $\{U_i\}_{1 \leq i \leq l}$ of $M$ by charts $(U_i,\phi_i)_{1 \leq i \leq l},$ and a partition of unity  $(\rho_i)_{1 \leq i \leq l}$ subordinated to $\mathcal{U} :=\{U_i\}_{1 \leq i \leq l}.$ 
For $\alpha = \{\alpha_i\}_{1 \leq i \leq l} \in \Omega^n(M)$ we define
\begin{equation}\label{integration}
\int_M \alpha := \sum_{i=1}^l \int_{\phi_i(U_i)} (\rho_i \circ \phi_i^{-1}) \alpha_i.   
\end{equation}

One can check that \eqref{integration} does not depend on choices of cover and partition of unity.

We are now ready to state the $L^{\infty}$ Poincar\'e duality. 

\begin{theorem}[$L^{\infty}$-Poincar\'e duality] \label{PD}
The pairing
\begin{align}
    \langle \cdot, \cdot \rangle: H_{L^{\infty},dR}^k(M) &\times H_{L^{\infty},dR}^{n-k}(M) \arr \R, \nonumber \\
    ([\alpha], [\beta]) &\mapsto \int_M \alpha \wedge \beta \label{pd-map}
\end{align} \label{poincare-duality}
is non-degenerate.
\end{theorem}

For the proof of Theorem \ref{PD} see \cite[Theorem 2.1]{Teleman}.


\section{Proof of Theorem \ref{main-thm}} \label{section-proof}

We start this section with the comparison of symplectic-Lipschitz and weak symplectic-Lipschitz manifolds.

\begin{proposition} \label{lip-symp-equiv}
Symplectic-Lipschitz manifolds are weak symplectic-Lipschitz. 
\end{proposition}

\begin{proof}
Let $M$ be a symplectic-Lipschitz manifold of dimension $2n.$
Let $\{(U_a, \phi_a)\}_{a \in \A}$ be the maximal symplectic-Lipschitz atlas of $(M, \omega).$ 
By Definition \ref{lip-symp-def}, $M$ is equipped with a well-defined 2-form $\omega = \{\omega_0\vert_{\phi_i(U_i)}\}_{1 \leq i \leq k},$ where 
$$\omega_0 = \sum_{i=1}^n dx_i \wedge dx_{i+n}$$ 
is the standard symplectic form on $\R^{2n}.$
Notice that $\omega$ is an $L^{\infty}$ 2-form since it has constant coefficients in every chart, and  
that $d\omega = 0.$ Hence $\omega$ is a closed flat 2-form.
The non-degeneracy follows from the fact that $\omega^{\wedge n}$ induces the standard volume form in every chart.
This completes the proof of Proposition \ref{lip-symp-equiv}.
\end{proof}

We now switch to the proof of Theorem \ref{main-thm}.

\begin{proof}[Proof of Theorem \ref{main-thm}]
Let $(M,\omega)$ be a closed weak symplectic-Lipschitz manifold of dimension $2n.$
Since $\omega$ is a closed flat 2-form it defines an element $[\omega]$ in $H_{L_{\infty},dR}^2(M).$ 
Then from the non-degeneracy of $\omega$ and \eqref{pd-map} we get that
\begin{align*}
    \langle [\omega^{\wedge k}], [\omega^{\wedge (n-k)}] \rangle &= \int_M \vol > 0,
\end{align*}
since $\omega^{\wedge n}$ is an $L^{\infty}$ volume form,
Hence by Theorem \ref{poincare-duality} ($L^{\infty}$ Poincar\'e duality) it follows that $[\omega^{\wedge k}]\neq 0.$ 
Therefore $H^{2k}(M,\R) \cong H_{L_{\infty},dR}^{2k}(M) \neq 0.$ This completes the proof of Theorem \ref{main-thm}.
\end{proof}

\end{document}